\documentclass{amsart}

\usepackage{amsmath, amsthm, amssymb, amstext, amsfonts}
\usepackage{enumerate}
\usepackage{graphicx}
\usepackage[T1]{fontenc} %needed for polic ogonek
\usepackage[applemac]{inputenc}

\usepackage{graphicx}
\usepackage{xcolor}
\usepackage{tikz}
\usetikzlibrary{arrows}

\usepackage[backref]{hyperref}
\hypersetup{
    colorlinks,
    linkcolor={red!60!black},
    citecolor={green!60!black},
    urlcolor={blue!60!black},
    pdftitle={},
    pdfauthor={}
}

\theoremstyle{plain}

\newtheorem{thm}{Theorem}[section]
\newtheorem{lem}[thm]{Lemma}

\newtheorem{prob}[thm]{Problem}
\newtheorem*{prob*}{Problem}
\newtheorem{conj}[thm]{Conjecture}

\theoremstyle{definition}

\newcommand{\N}{\ensuremath{\mathbb{N}}}

\newcommand{\cE}{\ensuremath{\mathcal{E}}}

\newcommand{\cV}{\ensuremath{\mathcal{V}}}

\newcommand{\sm}{\ensuremath{\smallsetminus}}

\newcommand{\Aut}{\textnormal{Aut}}

\newcommand{\sub}{\subseteq}

\def\td{tree-decom\-po\-si\-tion}
\def\ta{tree amalgamation}
\def\qt{quasi-tran\-si\-tive}
\def\qi{quasi-iso\-metric}
\def\qiy{quasi-iso\-me\-try}
\def\qis{quasi-iso\-me\-tries}

\def\lf{locally finite}

\def\tw{tree-width}

\newcommand{\comment}[1]{}

\def\?#1{\vadjust{\vbox to 0pt{\vss\vskip-8pt\leftline{%
     \llap{\hbox{\vbox{\pretolerance=-1
     \doublehyphendemerits=0\finalhyphendemerits=0
     \hsize16truemm\tolerance=10000\small
     \lineskip=0pt\lineskiplimit=0pt
     \rightskip=0pt plus16truemm\baselineskip8pt\noindent
     \hskip0pt        %(without this, the first word is never hyphenated!)
     #1\endgraf}\hskip7truemm}}}\vss}}}

\newenvironment{txteq*}
  {
    \begin{equation*}
    \begin{minipage}[c]{0.85\textwidth} % set width to 0.9 x textwidth
    \em                                % switch on emph
  }
  {\end{minipage}\end{equation*}\ignorespacesafterend}

\begin{document}

\title{Quasi-transitive $K_\infty$-minor free graphs}
\author{Matthias Hamann}
\address{Matthias Hamann, Department of Mathematics, University of Hamburg, Hamburg, Germany}

\date{}

\begin{abstract}
We prove that every \lf\ \qt\ graph that does not contain $K_\infty$ as a minor is \qi\ to some planar \qt\ \lf\ graph.
This solves a problem of Esperet and Giocanti and improves their recent result that such graphs are \qi\ to some planar graph of bounded degree.
\end{abstract}

\maketitle

\section{Introduction}

Recently, Esperet and Giocanti \cite{EG-QTGraphsWithoutMinor} proved a theorem for \qt\ graphs, where a graph is \emph{\qt} if its automorphism group acts on its vertex set with only finitely many orbits.
Before we state their theorem , let us briefly introduce \qis.
A graph $G$ is \emph{\qi} to another graph~$H$ if there exists $\gamma\geq 1$ and $c\geq 0$ and a map $\varphi\colon V(G)\to V(H)$ such that the following holds.
\begin{enumerate}[(i)]
\item $\frac{1}{\gamma}d_G(u,v)-c\leq d_H(\varphi(u),\varphi(v))\leq \gamma d_G(u,v)+c$ for all $u,v\in V(G)$ and
\item $d_H(w,\varphi(V(G)))\leq c$ for all $w\in V(H)$.
\end{enumerate}
Then $\varphi$ is a \emph{\qiy}.
If the constants $\gamma$ and~$c$ are important, we call $\varphi$ also a \emph{$(\gamma,c)$-\qiy} and say that $G$ and~$H$ are \emph{$(\gamma,c)$-\qi}.

Now we are able to state the theorem of Esperet and Giocanti.

\begin{thm}\label{thm_EG}\cite[Theorem 1.3]{EG-QTGraphsWithoutMinor}
Every \lf\ \qt\ graph that does not contain $K_\infty$ as a minor is \qi\ to some planar graph of bounded degree.
\end{thm}

Esperet and Giocanti proved their theorem as a first step towards a more general conjecture by Georgakopoulos and Papasoglu \cite{GP-MinorsMetricSpaces}.
In order to state their conjecture, let us introduce the notion of asymptotic minors.

For $K\in\N$, a graph $H$ is a \emph{$K$-fat minor} of a second graph $G$ if there exists a family $(B_v)_{v\in V(H)}$ of connected subsets of $V(G)$ and a family $(P_e)_{e\in E(H)}$ of paths in~$G$ such that
\begin{enumerate}[(1)]
\item for all $uv\in E(H)$, the path $P_{uv}$ intersects $\bigcup_{w\in V(H)} B_w$ in exactly its end vertices, one of which lies in~$B_u$, the other in~$B_v$,
\item $d(P_{uv},B_w)\geq K$ for all $uv\in E(H)$ and $w\in V(H)\sm\{u,v\}$,
\item $d(B_u,B_v)\geq K$ for all distinct $u,v\in V(H)$, and
\item $d(P_e,P_{e'})\geq K$ for all distinct $e,e'\in E(H)$.
\end{enumerate}
We call $H$ an \emph{asymptotic minor} of~$G$ if for every $K > 0$, $H$ is a $K$-fat minor of~$G$.
Now we can state Georgakopoulos' and Papasolgu's conjecture.

\begin{conj}\cite[Conjecture 9.3]{GP-MinorsMetricSpaces}\label{conj_GP}
Let $G$ be a \lf\ transitive graph.
Then either $G$ is \qi\ to a planar graph, or it contains every finite graph as an asymptotic minor.
\end{conj}

The obvious question regarding Conjecture~\ref{conj_GP} is whether we can ask the planar graph to be transitive, too.
Indeed, Esperet and Giocanti \cite[Section 6]{EG-QTGraphsWithoutMinor} raised the problem whether the planar graph in their theorem can be asked to be \qt, too.
We will prove that this is possible.
That is, we will prove the following theorem.

\begin{thm}\label{thm_main}
Every \lf\ \qt\ graph that does not contain $K_\infty$ as a minor is \qi\ to some planar \qt\ \lf\ graph.
\end{thm}

This result indicates that a possible positive solution of the following problem might be expectable.

\begin{prob}\label{prob_QI}
If $G$ is a \qt\ \lf\ graph \qi\ to a planar graph, then is $G$ \qi\ to a \qt\ \lf\ planar graph?
\end{prob}

Another hint that this might be true is that MacManus~\cite{M-AccessPlanarQI} recently proved the following analogous statement for finitely generated groups.

\begin{thm}\label{thm_MacManus}\cite[Corollary D]{M-AccessPlanarQI}
The following are equivalent for every finitely generated group $G$.
\begin{enumerate}[\rm (1)]
\item $G$ is \qi\ to a planar graph.
\item $G$ is \qi\ to a planar Cayley graph.
\end{enumerate}
\end{thm}

Furthermore, he proved a structural result for \qt\ \lf\ graphs that are \qi\ to planar graphs, see \cite[Corollary C]{M-AccessPlanarQI}, in terms of canonical \td s: the parts are either finite or \qi\ to complete Riemannian planes.
We refer to Section~\ref{sec_prelim} for the definition of (canonical) \td s.
This structural result might be useful for Problem~\ref{prob_QI}.

\section{Preliminaries}\label{sec_prelim}

Let $G$ be a graph.
A \emph{\td} of~$G$ is a pair $(T,\cV)$ of a tree~$T$, the \emph{decomposition tree}, and a family $\cV=(V_t)_{t\in V(T)}$ of vertex sets of~$G$, one for every $t\in V(T)$, such that
\begin{enumerate}[(T1)]
\item $V(G)=\bigcup_{v\in V(T)}V_t$,
\item for every $e\in E(G)$ there exists $t\in V(T)$ with $e\sub V_t$, and
\item $V_{t_1}\cap V_{t_2}\sub V_{t_3}$ for all $t_3$ on the $t_1$-$t_2$ path in~$T$.
\end{enumerate}
The sets $V_t$ are the \emph{parts} of the \td\ and the intersection $V_{t_1}\cap V_{t_2}$ for adjacent $t_1$ and $t_2$ are the \emph{adhesion sets}.
The \emph{adhesion} of $(T,\cV)$ is the supremum of the sizes of the adhesion sets.
The \emph{width} of $(T,\cV)$ is $\sup_{t\in V(T)}|V_t|-1$, seen as an element of $\N\cup\{\infty\}$, if all $V_t$ are finite and $\infty$ otherwise.
The \emph{\tw} of~$G$ is the minimum width among all \td s of~$G$.

If the automorphism group of~$G$ induces an action on the family $\cV$ and thereby also an action on~$T$ then we call the \td\ \emph{canonical}.

If $V_t$ is a part of $(T,\cV)$, then the subgraph of~$G$ induced by~$V_t$ together with all (possibly new) edges $uv$ for all distinct $u,v$ that lie in a common adhesion set in~$V_t$ is a \emph{torso} of $(T,\cV)$.

A \emph{separation} of~$G$ is a pair $(A,B)$ with $A,B\sub V(G)$ such that $A\cup B= V(G)$ and such that $e\sub A$ or $e\sub B$ for all edges of~$G$.
We call $|A\cap B|$ its \emph{order}.
The separation is \emph{tight} if there are components $C_A$ in $A\sm B$ and $C_B$ in $B\sm A$ with $N(C_A)=A\cap B=N(C_B)$.

For a \td\ $(T,\cV)$ and an edge $e\in E(T)$, the \emph{edge-separation} of~$e$ is the separation
\[
(\bigcup_{t\in V(T_1)}V_t,\bigcup_{t\in V(T_2)}V_t),
\]
where $T_1$ and $T_2$ are the two components of $T-e$.

The following result by Thomassen and Woess \cite[Corollary 4.3]{ThomassenWoess} was stated for transitive graphs, but its proof carries over almost verbatim to \qt\ graphs.

\begin{lem}\label{lem_TW4.3}\cite[Corollary 4.3]{ThomassenWoess}
Let $G$ be a connected \qt\ \lf\ graph and let $k\in\N$.
Then there are only finitely many $\Aut(G)$-orbits of tight separations of order~$k$.
\end{lem}

The major tool in our proof of Theorem~\ref{thm_main} is the following result by Esperet et al.~\cite{EGLD-QTAvoidingMinor}.

\begin{thm}\label{thm_EGL}\cite[Theorem 4.3]{EGLD-QTAvoidingMinor}
Let $G$ be a \qt\ \lf\ graph without $K_\infty$ as a minor and let $\Gamma$ be a group acting \qt ly on~$G$.
Then there exists $k\in\N$ and a $\Gamma$-invariant \td\ $(T, \cV)$ of adhesion
at most~$3$, and such that for every $t\in V(T)$ the torso of~$V_t$ is a minor of~$G$ that is either planar or has \tw\ at most~$k$ and such that $\Gamma_t$ acts \qt ly on that torso.
Furthermore, the edge-separations of $(T, \cV)$ are all tight.
\end{thm}

One-way infinite paths are \emph{rays} and two rays in a graph $G$ are \emph{equivalent} if, for every finite vertex set $S\sub V(G)$, both rays have all but finitely many vertices in the same component of $G-S$.
This is an equivalence relation whose equivalence classes are the \emph{ends} of~$G$.
An end is \emph{thick} if it contains infinitely many pairwise disjoint rays and it is \emph{thin} otherwise.
By a result of Halin~\cite{H-Endengrad}, for every thin end, there exists $n\in\N$ such that there are $n$ but not $n+1$ pairwise disjoint rays in that end.

Two ends are \emph{$k$-distinguishable} for some $k\in\N$ if there exists a vertex set $S$ of size at most~$k$ such that no component of $G-S$ contains all but finitely many vertices from rays from both ends.
A \td\ \emph{distinguishes} two ends \emph{efficiently} if there is an edge-separation $(A,B)$ such that all rays from one of the ends lie eventually in~$A$, all rays from the other end lie eventually in~$B$ and the ends are not $(|A\cap B|-1)$-distinguishable.
A \qt\ graph is \emph{accessible} if there exists $n\in\N$ such that every two ends are $n$-distinguishable.

The following is a special case of \cite[Theorem 7.3]{CHM-CanonicalTTD}.

\begin{thm}\label{thm_canonicalTD}
Let $G$ be a locally finite graph and let $k\in\N$.
Let $\cE$ be a set of ends of~$G$ that are pairwise $k$-distinguishable.
Then there is a canonical \td\ distinguishing all end in~$\cE$ efficiently.
\end{thm}

While the following statement follows from results about factorisations and \ta s of \qt\ graphs, we offer here a proof that avoids most of the definitions that we would need, if we conclude it from~\cite[Theorem 7.5]{HLMR}.

\begin{thm}\label{thm_canonicalTDFinite}
Let $G$ be a locally finite graph of finite \tw.
Then there exists a canonical \td\ of finite width distinguishing all ends of~$G$ efficiently.
\end{thm}

\begin{proof}
Since graphs of finite \tw\ $k\in\N$ are \qi\ to any decomposition tree of a \td\ of width~$k$ and since accessibility is preserved by \qis, those graphs are accessible and all ends are thin.
So let $(T,\cV)$ be a canonical \td\ distinguishing all ends of~$G$ efficiently.
We may assume that every edge-separation distinguishes some pair of ends efficiently.
In particular, there is an upper bound on the adhesion sets.
By Lemma~\ref{lem_TW4.3}, there are only finitely many orbits of tight separations of bounded order.
Thus, there are only finitely many orbits on~$E(T)$ and hence on~$V(T)$.
If we show that all parts are finite, then this implies that the \td\ has finite width.
So let us suppose that some part is infinite.
Since $(T,\cV)$ distinguishes all ends, there is a unique end in this part\footnote{An end $\omega$ \emph{lies} in a part~$V_t$ if some ray $R\in\omega$ meets $V_t$ infinitely often.} and hence also in this torso.
Since the stabiliser of that part acts \qt ly on the torso by a results of Esperet and Giocanti \cite[Lemma 3.13]{EGLD-QTAvoidingMinor}, it is a one-ended \qt\ graph.
By a result of Thomassen \cite[Proposition 5.6]{T-Hadwiger}, this end must be thick, a contradiction since all ends are thin.
Thus, all parts are finite, which finishes the proof as mentioned above.
\end{proof}

For a finite tree~$T$, we call a vertex of~$T$ \emph{central} if it is the middle vertex of a longest path in~$T$.
Similarly, an edge of~$T$ is \emph{central} if it is the middle edge of a longest path in~$T$.
Note that every finite tree has either a central vertex or a central edge and that this is always fixed the automorphism group of the tree.

\section{Proof of Theorem~\ref{thm_main}}\label{sec_proof}

Let $G$ be a \qt\ \lf\ graph that omits $K_\infty$ as a minor.
By Theorem~\ref{thm_EGL}, there exist $k\in\N$ and a canonical \td\ $(T,\cV)$ of~$G$ of adhesion at most~$3$ such that the torsos are minors of~$G$ and each torso is either planar or has \tw\ at most~$k$ and such that the stabiliser of each torso acts \qt ly on that torso.
Furthermore, the edge-separations of $(T,\cV)$ are tight.
Thus, there are only finitely many orbits of them by Lemma~\ref{lem_TW4.3} and hence there are only finitely many $\Aut(G)$-orbits on $V(T)$.

We distinguish three types of torsos (finite torsos, infinite torsos of \tw\ at most~$k$ and infinite planar torsos) and prepare them for our final \qiy: we find for each torso of the first two kinds \qis\ to planar \qt\ \lf\ graphs and, in the last situation, we have to prepare them such that separations of order~$3$ whose separator is also an adhesion set in $(T,\cV)$ does not leave three distinct components.
We do this by adding additional separators of size~$1$.

If there are finite torsos, then there is an upper bound $B_1$ on the number of vertices in each such torso as there are only finitely many $\Aut(G)$-orbits on $V(T)$.
Thus, each of those torsos is $(1,B_1)$-\qi\ to a single vertex.

Let us now consider an infinite torso $H_t$ of \tw\ at most~$k$.
Since it is \lf, $H_t$ has a canonical \td\ of finite width.
Again, since there are only finitely many $\Aut(G)$-orbits on $V(T)$, there exists an upper bound $B_2$ on the width of the canonical \td s of such torsos.
Let $(T_t,\cV_t)$ be a canonical \td\ of~$H_t$ of width at most~$B_2$ distinguishing all ends and such that all of its edge-separations are tight, which exists by Theorem~\ref{thm_canonicalTDFinite}.
Since there are only finitely many orbits on $V(T_t)$ under the stabiliser of $H_t$ by the same argument that we have only finitely many $\Aut(G)$-orbits on $V(T)$, there exists an upper bound $B_3$ on the diameter of the parts of $(T_t,\cV_t)$ and an upper bound $B_4$ on the number of parts that contain a vertex~$v$.
Again, we may assume that these bounds $B_3$ and~$B_4$ hold for all torsos of this type, i.\,e.\ all infinite torsos of \tw\ at most~$k$.
Thus, any map that maps each vertex $u$ of~$H_t$ to some $s\in V(T_t)$ such that $u$ lies in the part of~$s$ is a $(1, B_3B_4)$-\qiy\ from $H_t$ to~$T_t$.
Since every adhesion set $S$ of $(T,\cV)$ in~$H_t$ is a clique and thus must lie in some common part of $(T_t,\cV_t)$, there exists a non-empty subtree $T_t^S$ of $T_t$ all of whose parts contain~$S$.
As all edge-separations of $(T_t,\cV_t)$ are tight, Lemma~\ref{lem_TW4.3} implies that every $T_t^S$ is finite.
So it has a central vertex $v_S$ or a central edge~$e_S$.

For every infinite planar torso $H_t$ and every adhesion set $S$ in~$H_t$ of size~$3$, there are at most two components $C$ of $H_t-S$ with $N(C)=S$, since $H_t$ is planar and thus does not contain $K_{3,3}$ as a minor.
Let $(T_t,\cV_t)$ be a canonical \td\ of adhesion at most~$3$ distinguishing all $3$-distinguishable ends of~$H_t$ such that all of its edge-separations are tight.
This exists by Theorem~\ref{thm_canonicalTD}.
We contract all edges whose edge-separations do not have one of the adhesion sets of size~$3$ from $(T,\cV)$ as separator and join their parts.
Thereby, we obtain a \td\ $(T_t',\cV_t')$ that has as adhesion sets only adhesion sets of size~$3$ that are also adhesion sets in $(T,\cV)$.
Note that the torsos are the subgraphs of~$H_t$ induced by the parts.
Let $G_s$ be a torso of $(T_t',\cV_t')$.
If there is an adhesion set $S\sub V(G_s)$ of $(T,\cV)$ that is not an adhesion set in $(T_t',\cV_t')$, then $G_s-S$ has a unique infinite component that is completely attached to~$S$, i.\,e.\ has all vertices from $S$ in its neighbourhood, and perhaps one finite component.
We delete that finite one.
By doing this for all choices of~$S$, we obtain a new graph~$G_s'$.
As there are only finitely many orbits on the adhesion sets in $(T,\cV)$, there exists $B_5$ such that $G_s$ is $(1,B_5)$-\qi\ to~$G_s'$ for all choices of~$G_s$.

Now we are ready to define the graph $H$ that will be \qt, \lf, planar and \qi\ to~$G$.
For that, we take the disjoint union $H'$ of the following graphs:
\begin{enumerate}[(i)]
\item one vertex $x_S$ for every adhesion set $S$ in $(T,\cV)$;
\item one vertex $x_t$ for every finite torso $H_t$ of $(T,\cV)$;
\item one copy of the decomposition tree $T_t$ for every infinite torso of \tw\ at most~$k$ and
\item the disjoint union of all graphs $G_s$ obtained from torsos $G_s'$ in the \td\ $(T_t',\cV_t')$ of the infinite planar torsos of $(T,\cV)$ that do not have \tw\ at most~$k$.
\end{enumerate}
In order to form the graph $H$, we add some edges to~$H'$:
\begin{enumerate}[(i)]\setcounter{enumi}{4}
\item an edge $x_Sx_t$ for all adhesion sets $S$ and finite torsos $H_t$ with $S\sub V_t$;
\item an edge $x_Sv_S$ or two edges from $x_S$ to the vertices incident with $e_S$ for all adhesion sets $S$ and infinite torsos of \tw\ at most~$k$ that contain~$S$ and
\item edges from all $s\in S$ to $x_S$ for all adhesion sets $S$ in $(T,\cV)$ and the graphs $G_s$ that contain~$S$.
\end{enumerate}
The resulting graph is denoted by~$H$.
By construction, $G$ is connected and $(1,B)$-\qi\ to~$H$, where $B$ is the maximum of $B_1$, $B_3B_4$ and~$B_5$.
Since we made no choices during the construction of~$H$ that were not invariant under the automorphisms, the automorphism group of~$G$ acts on~$H$.
By the choices during the construction, the stabiliser of each torso of $(T,\cV)$ still acts \qt ly on the graph that replaces this torso and as a result, $H$ is a \qt\ graph.
Obviously, it is \lf.
Since all components in~$H'$ are planar and since the vertices $x_S$ are $1$-separators and attached to either at most two adjacent vertices in a component of~$H'$ or to all vertices from the adhesion set $S$ of $(T,\cV)$ whose removal from each component of~$H'$ leave exactly one component with all of~$S$ in its neighbourhood, we obtain that $H$ is planar, too.
This finishes the proof of Theorem~\ref{thm_main}.\qed

\end{document}